\documentclass[11pt]{article}

\usepackage[T1]{fontenc}
\usepackage[utf8]{inputenc}
\usepackage{lmodern}
\usepackage[margin=1.02in]{geometry}
\usepackage{amsmath,amssymb,amsthm,mathtools}
\usepackage{microtype}
\usepackage{enumitem}
\usepackage[hidelinks]{hyperref}
\hypersetup{
  pdftitle={The Hayman--Wu constant is pi squared},
  pdfauthor={Paata Ivanisvili},
  pdfsubject={The sharp Hayman--Wu theorem for straight lines},
  pdfkeywords={Hayman--Wu theorem, conformal mapping, univalent function, Schwarz reflection, Herglotz representation}
}

\allowdisplaybreaks
\setlist[enumerate]{itemsep=0.3em,topsep=0.3em}

\newtheorem{theorem}{Theorem}[section]
\newtheorem{lemma}[theorem]{Lemma}
\theoremstyle{remark}
\newtheorem{remark}[theorem]{Remark}

\newcommand{\D}{\mathbb D}
\newcommand{\Hh}{\mathbb H}
\newcommand{\T}{\mathbb T}
\newcommand{\R}{\mathbb R}
\newcommand{\C}{\mathbb C}
\newcommand{\Rhat}{\widehat{\mathbb R}}
\newcommand{\Haus}{\mathcal H^1}

\newcommand{\HW}{C_{\mathrm{HW}}}

\title{The Hayman--Wu constant is $\pi^2$}
\author{Paata Ivanisvili\thanks{Department of Mathematics, University of California,
Irvine, 510C Rowland Hall, Irvine, CA 92697-3875, USA.
Email: \href{mailto:pivanisv@uci.edu}{\texttt{pivanisv@uci.edu}}.}}
\date{August 2026}

\begin{document}
\maketitle

\begin{abstract}
We show that for every conformal map $\phi:\mathbb D\to\Omega\subsetneq\mathbb C$
from the unit disk onto a simply connected proper domain $\Omega$, the length of
$\phi^{-1}(\Omega\cap L)$ is at most $\pi^2$ for every line $L$.
\end{abstract}



\section{Introduction}

For a simply connected proper domain $\Omega\subsetneq\C$, a conformal map
$\phi:\D\to\Omega$, and a line $L\subset\C$, set
\[
  \mathcal L(\phi,L)
  :=\Haus\!\left(\phi^{-1}(L\cap\Omega)\right)
\]
and let
\[
  \HW:=\sup\left\{
  \mathcal L(\phi,L):
  \substack{\Omega\subsetneq\C\text{ simply connected},\quad
  \phi:\D\to\Omega\text{ conformal},\\
  L\subset\C\text{ a line}}
  \right\}.
\]

\begin{theorem}\label{thm:main}
Let $\Omega\subsetneq\C$ be simply connected, let $\phi:\D\to\Omega$ be
conformal, and let $L$ be a line. Then
\begin{equation}\label{eq:main}
  \Haus\!\left(\phi^{-1}(L\cap\Omega)\right)\le \pi^2.
\end{equation}
Consequently, together with {\O}yma's lower construction,
\[
  \HW=\pi^2.
\]
\end{theorem}

Hayman and Wu proved that $\HW<\infty$ in \cite{HaymanWu}. {\O}yma obtained
the elementary upper bound $4\pi$ in \cite{OymaHarmonic}, Rohde proved
$\HW<4\pi$ in \cite{Rohde}, and {\O}yma constructed examples with length
tending to $\pi^2$ in \cite{OymaConstant}. Brown Flinn proved the sharp
bound when the line is contained in the image domain \cite{BrownFlinn}, and
Crane obtained it for a broader reflected-component class \cite{Crane}.

Our proof follows the reflected-domain reduction of {\O}yma, in the form
presented by Garnett and Marshall \cite[Chapter~I, Section~5]{GarnettMarshall}.
The new input is the contact-selector estimate, Lemma~\ref{lem:selector}.
Here is the argument in outline. After an inner approximation and a Euclidean
motion, assume that $\Omega$ is an analytic Jordan domain and that $L=\R$.
Let
\[
  \Omega^*=\{\overline z:z\in\Omega\},
  \qquad U=\Omega\cap\Omega^*.
\]
For each component $V$ of $U$ meeting $\R$, choose a symmetric conformal map
\[
  G_V:\Hh\to V,
  \qquad G_V(i\R_+)=V\cap\R.
\]
From every reflected pair $x,-x\in\partial\Hh$, select one parameter whose
image under $G_V$ lies on the original boundary $\partial\Omega$. The selected
set $E_V$ is a finite union of intervals. If $f=\phi^{-1}$ and
$F_V=f\circ G_V$, then
\[
  \int_0^\infty |F_V'(iy)|\,dy
  \le \frac\pi2\int_{E_V}|F_V'(x)|\,dx.
\]
The two sides are the lengths of $f(V\cap\R)$ and of a selected boundary set
$f(A_V)\subset\T$. Distinct reflected components use disjoint boundary arcs,
modulo endpoints, and therefore
\[
  \sum_V\Haus(f(A_V))\le2\pi.
\]
Summing gives $\pi^2=(\pi/2)(2\pi)$.

Throughout,
\[
  \Hh=\{z:\operatorname{Im}z>0\},
  \qquad \T=\partial\D,
  \qquad \Rhat=\R\cup\{\infty\}.
\]
All nonnegative integrals below are understood in the extended sense until
finiteness has been established.

\section{The contact-selector estimate}

\begin{lemma}[Contact-selector lemma]\label{lem:selector}
Let $F:\Hh\to\D$ be univalent, and let $E\subset\R$ be a finite union of
pairwise disjoint open intervals, possibly unbounded. Assume that
\begin{enumerate}[label=\textup{(\roman*)},leftmargin=2.4em]
\item for almost every $x>0$, exactly one of $x,-x$ belongs to $E$;
\item $F$ extends holomorphically across every interval comprising $E$;
\item $|F|=1$ on $E$.
\end{enumerate}
Then
\begin{equation}\label{eq:selector-main}
  \int_0^\infty |F'(iy)|\,dy
  \le \frac\pi2\int_E|F'(x)|\,dx.
\end{equation}
\end{lemma}

\begin{proof}
We may assume
\begin{equation}\label{eq:Afinite}
  A:=\int_E|F'(x)|\,dx<\infty.
\end{equation}

\smallskip
\noindent\emph{Factorization and boundary positivity.}
The function $F$ has at most one zero. If $F(a)=0$, where $a=u+iv$ and
$v>0$, put
\[
  B=B_a,\qquad B_a(z)=\frac{z-a}{z-\overline a};
\]
otherwise put $B\equiv1$. Schwarz--Pick gives $|F|\le|B_a|$ in the first
case. Hence in either case
\begin{equation}\label{eq:FBQ}
  F=BQ,
  \qquad Q\ne0,
  \qquad |B|,|Q|\le1.
\end{equation}
Write $Q=e^{ih}$. Then $\operatorname{Im}h\ge0$, and either $h$ is real
constant or $h:\Hh\to\Hh$ is a Pick function.

On every interval contained in $E$, the function $Q$ extends holomorphically,
is nonvanishing, and is unimodular. A local logarithm therefore extends $h$
holomorphically through that interval, with real boundary values. The Herglotz representation \cite[Chapter~I]{Duren} has the form
\begin{equation}\label{eq:Herglotz}
  h(z)=c+\alpha z+
  \int_{\R}\left(\frac1{t-z}-\frac{t}{1+t^2}\right)d\mu(t),
\end{equation}
where $c\in\R$, $\alpha\ge0$, and $\mu$ is positive with
$\int(1+t^2)^{-1}\,d\mu(t)<\infty$. By the standard Stieltjes inversion
formula, holomorphic continuation with real boundary values through an
interval implies that $\mu$ has no mass there. Consequently, for $x\in E$,
\begin{equation}\label{eq:hprime-boundary}
  h'(x)=\alpha+\int_{\R}\frac{d\mu(t)}{(t-x)^2}\ge0.
\end{equation}

If the zero $a=u+iv$ is present, set
\[
  p(x)=\frac{2v}{(x-u)^2+v^2};
\]
otherwise set $p=0$. Since $B'(x)/B(x)=ip(x)$ on $\R$, we obtain the exact
identity
\begin{equation}\label{eq:no-cancellation}
  |F'(x)|=p(x)+h'(x),
  \qquad x\in E.
\end{equation}
The selector has infinite Lebesgue measure. Thus \eqref{eq:Afinite},
\eqref{eq:hprime-boundary}, and \eqref{eq:no-cancellation} force
\begin{equation}\label{eq:alpha-zero}
  \alpha=0.
\end{equation}

\smallskip
\noindent\emph{The zero-free factor.}
Schwarz--Pick for $h$ and \eqref{eq:Herglotz} give
\[
  |Q'(iy)|
  \le\frac{\operatorname{Im}h(iy)}{y}
  =\int_{\R}\frac{d\mu(t)}{t^2+y^2}.
\]
Hence, by Tonelli,
\begin{equation}\label{eq:Qvertical}
  \int_0^\infty|Q'(iy)|\,dy
  \le\frac\pi2\int_{\R}\frac{d\mu(t)}{|t|}.
\end{equation}
On the other hand, the selector property gives, for every $t\ne0$,
\begin{equation}\label{eq:kernel}
  \int_E\frac{dx}{(x-t)^2}\ge\frac1{|t|}.
\end{equation}
Indeed, for $t>0$ each pair $x,-x$ contributes one of
$(x-t)^{-2},(x+t)^{-2}$, and the smaller one integrates over $x>0$ to $1/t$;
the case $t<0$ is symmetric, while for $t=0$ the inner integral is infinite.
Therefore
\begin{equation}\label{eq:Qestimate}
  \int_Eh'(x)\,dx
  =\int_{\R}\left(\int_E\frac{dx}{(x-t)^2}\right)d\mu(t)
  \ge\int_{\R}\frac{d\mu(t)}{|t|},
\end{equation}
and
\begin{equation}\label{eq:Qfinal}
  \int_0^\infty|Q'(iy)|\,dy
  \le\frac\pi2\int_Eh'(x)\,dx.
\end{equation}

\smallskip
\noindent\emph{The possible zero.}
If $a=u+iv$ is present, then
\[
  |B_a'(iy)|=\frac{2v}{u^2+(y+v)^2}.
\]
For $u\ne0$, writing $r=v/|u|$, we have
\begin{equation}\label{eq:Bvertical}
  \int_0^\infty|B_a'(iy)|\,dy=2r\arctan(1/r),
\end{equation}
whereas the selector property yields
\begin{equation}\label{eq:Bboundary}
  \int_Ep(x)\,dx
  \ge\int_0^\infty\min\{p(x),p(-x)\}\,dx
  =2\arctan r.
\end{equation}
The elementary inequality
\[
  2r\arctan(1/r)\le\pi\arctan r
\]
follows on writing $r=\tan s$ and using
$\cos s\ge2(\pi/2-s)/\pi$ and $\sin s\le s$, both for
$0\leq s \leq \pi/2$. Hence
\begin{equation}\label{eq:Bfinal}
  \int_0^\infty|B_a'(iy)|\,dy
  \le\frac\pi2\int_Ep(x)\,dx.
\end{equation}
For $u=0$, the two sides before multiplication by $\pi/2$ are respectively
$2$ and $\pi$, so \eqref{eq:Bfinal} still holds.

Finally, $F=BQ$ and $|B|,|Q|\le1$ imply $|F'|\le|B'|+|Q'|$ on $i\R_+$.
Combining \eqref{eq:Qfinal}, \eqref{eq:Bfinal}, and
\eqref{eq:no-cancellation} proves \eqref{eq:selector-main}.
\end{proof}

\begin{remark}[Sharpness]\label{rem:sharpness}
Let $F=B_{u+iv}$ with $u>0$ and take $E=(-\infty,0)$. As $u/v\to\infty$,
\[
  \int_0^\infty|F'(iy)|\,dy\sim\frac{\pi v}{u},
  \qquad
  \int_E|F'(x)|\,dx\sim\frac{2v}{u}.
\]
Thus the constant $\pi/2$ in Lemma~\ref{lem:selector} is sharp.
\end{remark}

\section{The analytic Jordan case}

Assume that $\Omega$ is a bounded analytic Jordan domain, let
$f:\Omega\to\D$ be conformal, and put
\[
  \Omega^*=\{\overline z:z\in\Omega\},
  \qquad U=\Omega\cap\Omega^*.
\]
Let $L_k$ be the components of $\Omega\cap\R$, and let $V_k$ be the component
of $U$ containing $L_k$. We use the same standard planar facts as in the
classical proof of Garnett and Marshall \cite[pp.~23--25]{GarnettMarshall}:
each $V_k$ is a Jordan domain symmetric about $\R$, one has
$V_k\cap\R=L_k$, and, if $j\ne k$,
\begin{equation}\label{eq:classical-separation}
  \partial V_j\cap\partial V_k\subset\R,
  \qquad
  \#(\partial V_j\cap\partial V_k)\le1.
\end{equation}
Only finitely many $L_k$ occur. Indeed, unless $\Omega=\Omega^*$, the two
analytic Jordan curves $\partial\Omega$ and $\partial\Omega^*$ meet only
finitely many times.

By symmetry and the Carath\'eodory theorem, for every $k$ there is a conformal
map
\begin{equation}\label{eq:Gk}
  G_k:\Hh\to V_k
\end{equation}
which extends homeomorphically from $\Rhat$ onto $\partial V_k$ and satisfies
\begin{equation}\label{eq:Gk-symmetry}
  G_k(i\R_+)=L_k,
  \qquad
  G_k(-\overline z)=\overline{G_k(z)}.
\end{equation}

We now construct the selector. If $\Omega\ne\Omega^*$, cut $\Rhat$ at
$\infty$ and at the preimages under $G_k$ of
$\partial\Omega\cap\partial\Omega^*$; if $\Omega=\Omega^*$, cut only at
$\infty$. Add the reflected cut points and $0$. The remaining real intervals
form a finite symmetric family. On each such interval, $G_k$ maps one-to-one
onto an analytic boundary arc lying entirely in $\partial\Omega$ or entirely
in $\partial\Omega^*$, and the Schwarz reflection principle extends $G_k$
holomorphically across the interval.

For every reflected pair of these intervals choose the one mapped into
$\partial\Omega$; when both are, choose the positive one. This is possible
because for every nonexceptional $x>0$,
\[
  G_k(-x)=\overline{G_k(x)},
  \qquad
  G_k(x)\in\partial\Omega\cup\partial\Omega^*.
\]
Let $E_k$ be the union of the selected intervals and put
\[
  A_k=G_k(E_k)\subset\partial V_k\cap\partial\Omega,
  \qquad
  F_k=f\circ G_k:\Hh\to\D.
\]
The Carath\'eodory theorem and Schwarz reflection show that $F_k$ extends
holomorphically across every selected interval and satisfies $|F_k|=1$ there.
Lemma~\ref{lem:selector} therefore applies.

Using the standard arclength formula for injective parametrizations, we get
\begin{align}
  \Haus(f(L_k))
  &=\int_0^\infty|F_k'(iy)|\,dy \notag\\
  &\le\frac\pi2\int_{E_k}|F_k'(x)|\,dx
   =\frac\pi2\Haus(f(A_k)).
  \label{eq:component-estimate}
\end{align}
By \eqref{eq:classical-separation}, the sets $A_k$ are pairwise disjoint
modulo at most one point. Since $f:\partial\Omega\to\T$ is a homeomorphism,
\[
  \sum_k\Haus(f(A_k))\le\Haus(\T)=2\pi.
\]
The intervals $L_k$ are disjoint and exhaust $\Omega\cap\R$. Summing
\eqref{eq:component-estimate} gives
\begin{equation}\label{eq:analytic-Jordan}
  \Haus(f(\Omega\cap\R))\le\pi^2.
\end{equation}

\section{Completion of the proof}

\begin{proof}[Proof of Theorem~\ref{thm:main}]
A Euclidean motion reduces the line to $\R$. Write $f=\phi^{-1}$ and, for
$0<r<1$, set
\[
  \Omega_r=\phi(r\D),
  \qquad
  f_r=\frac fr:\Omega_r\to\D.
\]
The domain $\Omega_r$ is a bounded analytic Jordan domain. Applying
\eqref{eq:analytic-Jordan} to $f_r$ gives
\[
  \Haus(f(\Omega_r\cap\R))\le r\pi^2.
\]
Moreover,
\[
  f(\Omega_r\cap\R)
  =f(\Omega\cap\R)\cap r\D,
\]
and these Borel sets increase to $f(\Omega\cap\R)$ as $r\uparrow1$.
Continuity from below of $\Haus$ therefore gives
\[
  \Haus(f(\Omega\cap\R))\le\pi^2.
\]
Finally, {\O}yma's examples \cite{OymaConstant} show that the constant cannot
be decreased.
\end{proof}

\section*{Acknowledgments}

The author acknowledges partial support from the NSF CAREER grant
DMS-2152401, NSF grant DMS-2554183, a Simons Fellowship, and a Humboldt
Research Fellowship for Experienced Researchers. The author acknowledges the
use of AI tools.


\begin{thebibliography}{9}

\bibitem{BrownFlinn}
B. Brown Flinn,
\emph{Hyperbolic convexity and level sets of analytic functions},
Indiana Univ. Math. J. \textbf{32} (1983), no. 6, 831--841.

\bibitem{Crane}
E. Crane,
\emph{A note on the Hayman--Wu theorem},
Comput. Methods Funct. Theory \textbf{8} (2008), no. 2, 615--624.

\bibitem{Duren}
P. L. Duren,
\emph{Theory of $H^p$ Spaces},
Pure and Applied Mathematics, vol. 38, Academic Press, New York, 1970.

\bibitem{GarnettMarshall}
J. B. Garnett and D. E. Marshall,
\emph{Harmonic Measure},
New Mathematical Monographs, vol. 2, Cambridge University Press,
Cambridge, 2005.

\bibitem{HaymanWu}
W. K. Hayman and J. M. G. Wu,
\emph{Level sets of univalent functions},
Comment. Math. Helv. \textbf{56} (1981), 366--403.

\bibitem{OymaHarmonic}
K. {\O}yma,
\emph{Harmonic measure and conformal length},
Proc. Amer. Math. Soc. \textbf{115} (1992), no. 3, 687--689.

\bibitem{OymaConstant}
K. {\O}yma,
\emph{The Hayman--Wu constant},
Proc. Amer. Math. Soc. \textbf{119} (1993), no. 1, 337--338.

\bibitem{Rohde}
S. Rohde,
\emph{On the theorem of Hayman and Wu},
Proc. Amer. Math. Soc. \textbf{130} (2002), no. 2, 387--394.

\end{thebibliography}
\end{document}